\documentclass[a4paper,11pt]{amsart}
\usepackage{amsfonts,amssymb,epsfig,latexsym}
\usepackage{amsmath}
\usepackage[latin1]{inputenc}
\usepackage{color,layout}
\usepackage{ifthen}

\usepackage{pdfsync}
\usepackage{graphicx}
\usepackage{color}

\newtheorem{theorem}{Theorem}[section]
\newtheorem{lemma}[theorem]{Lemma}
\newtheorem{proposition}[theorem]{Proposition}

\newtheorem{remark}[theorem]{Remark}

\def\square{\hbox{\vrule\vbox{\hrule\phantom{o}\hrule}\vrule}}

\def\R{\mathbb {R}}
\def\C{\mathbb {C}}
\def\N{\mathbb {N}}

\def\tendsto{\rightarrow}
\def\re{\mathop{\rm Re}\nolimits}
\def\im{\mathop{\rm Im}\nolimits}
\def\la{\langle}
\def\ra{\rangle}
\def\O{\mathcal O}
\newcommand{\be}{\begin{equation}}
\newcommand{\ee}{\end{equation}}

\numberwithin{equation}{section}

\begin{document}

\title[Lower bound of resonances for Helmhotz resonator]{Optimal lower bound of the resonance widths for a Helmhotz tube-shaped resonator}
 
\begin{abstract}
The study of the resonances of the Helmholtz resonator has been broadly described in previous works \cite{HM}. Here, for a simple tube-shaped two dimensional resonator, we can perform a careful analysis of the transition zone where oscillations start to appear. In that way, we obtain an optimal lower bound of the width of the resonances.
\end{abstract}

\author{Andr\'e Martinez${}^1$  \& Laurence N\'ed\'elec${}^2$}


\maketitle

\addtocounter{footnote}{1}

\footnotetext{Universit\`a di Bologna, Dipartimento di
Matematica, Piazza di Porta San Donato 5, 40127 Bologna,
Italy. Partly supported by Universit\`a di Bologna, Funds
for Selected Research Topics and Founds for Agreements with
Foreign Universities}
\addtocounter{footnote}{1}
\footnotetext{Stanford University, Department of Mathematics, building 380, Stanford, California 94305}

\setcounter{section}{0}

\section{Introduction}

Historically, the Helmholtz resonator was conceived and built by Herman von Helmholtz in order to study  vibrations and their receiving  by human beings.

From a physical point of view, it consists of a bounded cavity (the chamber) connected to the exterior by a thin tube (the neck of the chamber). When air is forced into the chamber through the aperture of the neck, a vibrational procedure (a sound) starts due to pressure changes, with a frequency dependent on the shape of the chamber.

Mathematically, this phenomenon is described by the resonances of the Dirichlet Laplacian $-\Delta_\Omega$ on the domain $\Omega$ consisting in the union of the chamber, the neck and the exterior (see Figure  \ref{fig1}).

\begin{figure}[h]
\label{fig1}

\vspace*{0ex}
\centering
\includegraphics[width=0.6\textwidth,angle=0]{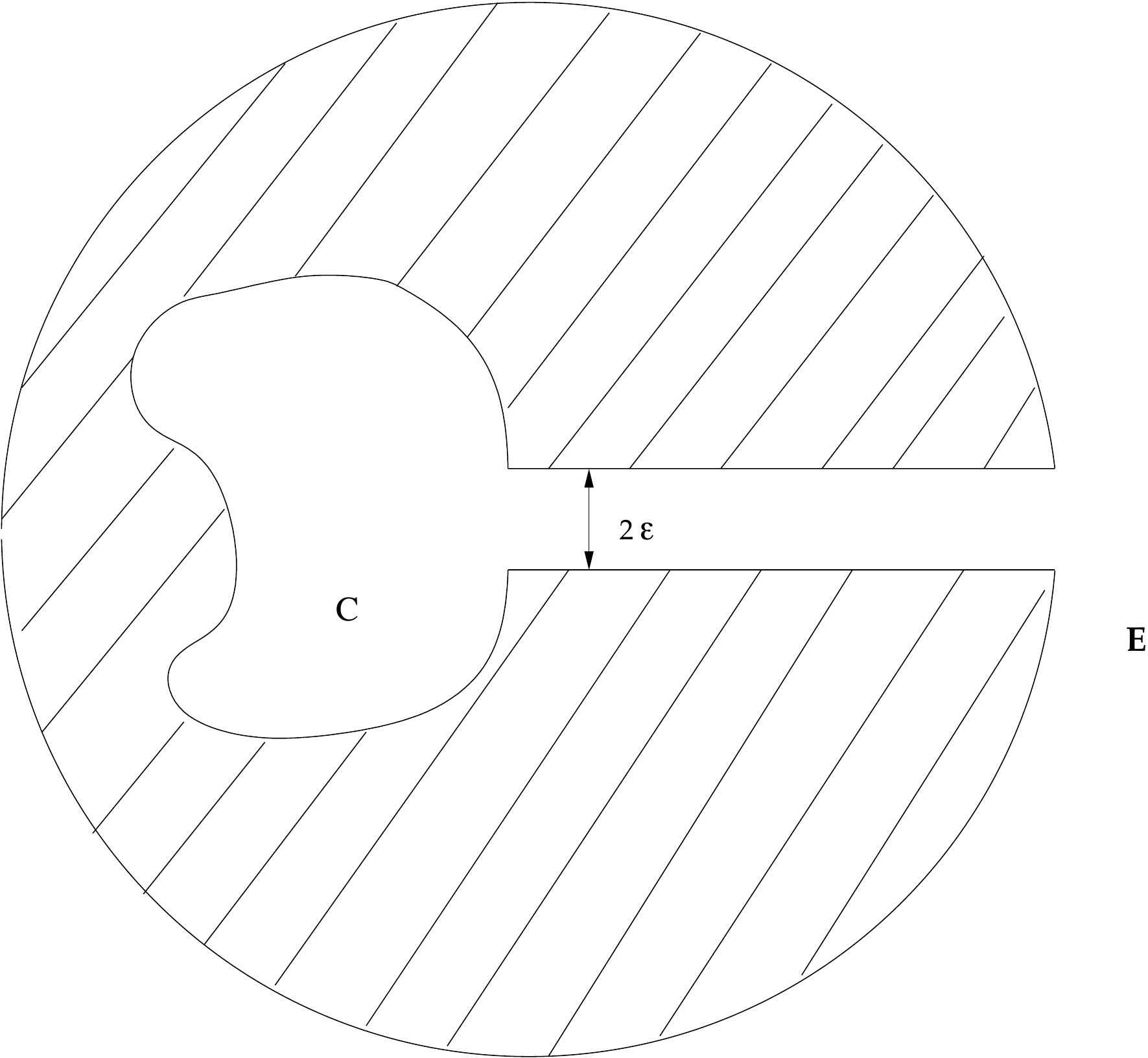}

\vspace*{0ex}
\caption{The Helmholtz resonator}
\end{figure}

More precisely, the resonances are defined as the eigenvalues of a complex deformation of $-\Delta_\Omega$;  their real part corresponds to the frequency, while their imaginary part corresponds to the inverse of the half-life of the vibrational mode.
Therefore, it is of physical interest to compute as precisely as possible both of these quantities. 

An efficient approach consists in an asymptotic study of this problem when the width $\varepsilon$ of the neck is  arbitrarily small. Indeed, in this situation, the work \cite{HM} shows that the frequencies are closed those of the chamber (that is, to the real eigenvalues of the Dirichlet Laplacian on the cavity), and gives an exponentially small upper bound on the absolute values of the imaginary part (width) of the resonances. However, no lower bound is known in the general situation.

Here, we focus on obtaining such a lower bound (of the same order of magnitude as the corresponding upper bound), in the particular case of a two dimensional tube-shaped resonator (see figure 2).

\begin{figure}[h]
\label{fig:es}
\vspace*{0ex}
\centering
\includegraphics[width=0.75\textwidth,angle=0]{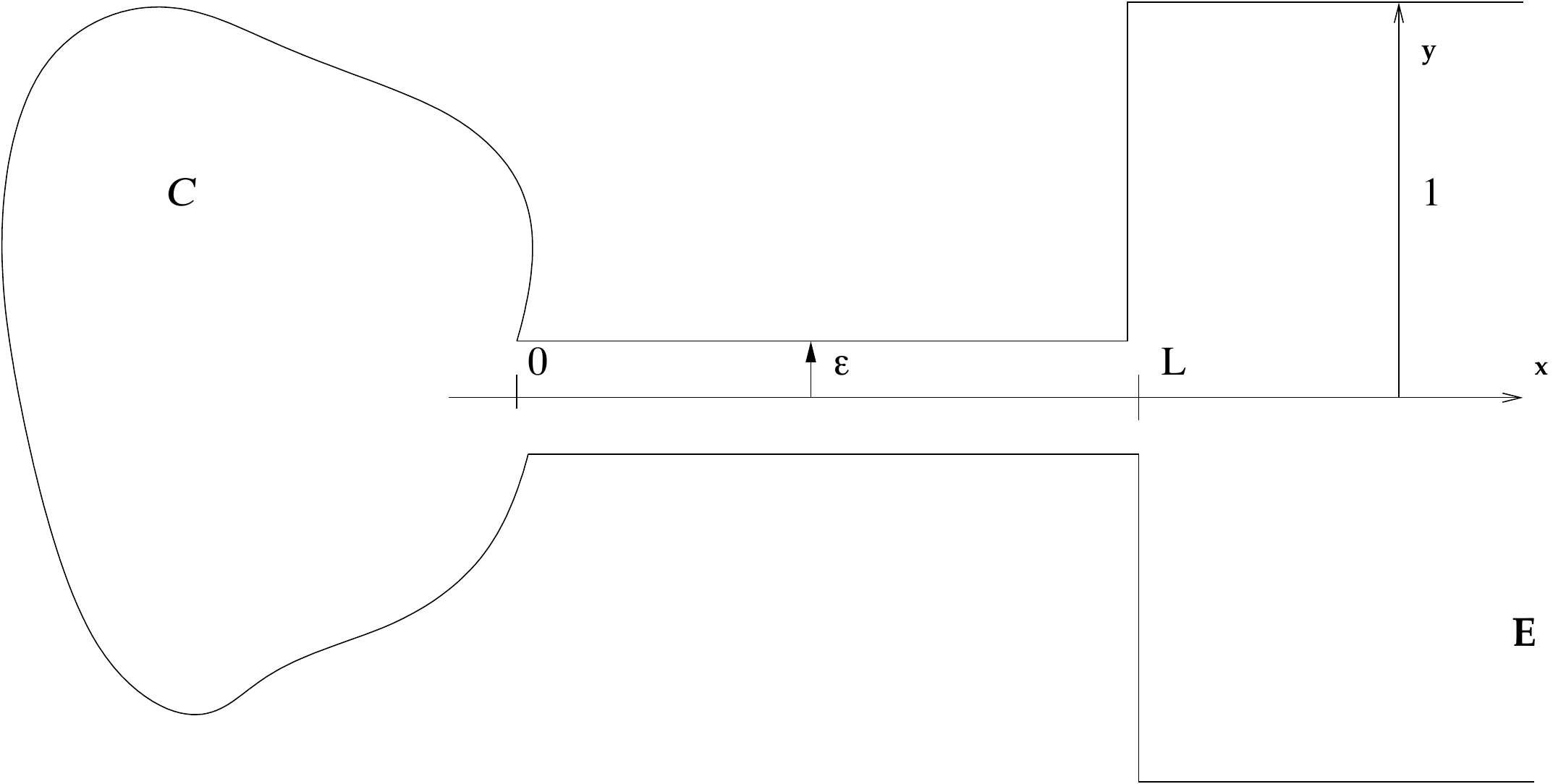}
\vspace*{-20ex}
\vspace {3cm}
\caption{The tube-shaped resonator}
\end{figure}

To our knowledge, and apart from the one-dimensional case (see, e.g., \cite{Ha, FL}), the only available results on  lower bounds concerning exponentially small widths of resonances are \cite{Bu, HS}, and only \cite{HS} is optimal (see also \cite{FLM} for some generalization).

\section{Geometrical description and results}
In $\R^2$, we consider a Helmholtz resonator consisting of a regular bounded open set  $\mathcal C$ (the cavity), connected to an unbounded domain $\mathbf E$ through a thin straight tube ${\mathcal T}(\varepsilon)$ (the neck), where $\varepsilon >0$ is a small parameter  which equals half the diameter of the tube (see figure 2).

More precisely, we assume that the Euclidean coordinates $(x,y)$ of $\R^2$ can be chosen in such a way that, for some $L, \delta >0$  independent of $\varepsilon$, one has,
$$ \begin{array}{l}
{\mathbf E}= (L, +\infty)\times (-1,1); \\
0\in \partial{\mathcal C} ;\\
 \overline{\mathcal C}\cap \left( ( [0,L]\times\{0\})\, \cup  \,\overline{\mathbf E}\right)= \emptyset;\\
{\mathcal T}(\varepsilon)=\left([-\delta, L]\times (-\varepsilon, \varepsilon)\right) \cap \left( \R^2\backslash {\mathcal C}\right).\end{array}
$$

Here, $\partial{\mathcal C}$  stands for the  boundary of ${\mathcal C}$.

In particular, as $\varepsilon \tendsto 0_+$, the resonator $\Omega(\varepsilon):={\mathcal C}\cup{\mathcal T}(\varepsilon)\cup{\mathbf E}$ collapes into $\Omega_0:={\mathcal C}\cup [0,M_0]\cup{\mathbf E}$, with $M_0=(L,0)$. 

Let $P_\varepsilon=-\Delta_{\Omega(\varepsilon)}$ be the Dirichlet Laplacian on $\Omega (\varepsilon)$. 
\par\noindent Let also $P_{\mathcal{C}(\varepsilon)}=-\Delta_{\mathcal{C}\cup \mathcal{T}(\varepsilon)}$ be the Dirichlet Laplacian on $\mathcal{C}(\varepsilon):=\mathcal{C}\cup \mathcal{T}(\varepsilon)$. \par\noindent  Finally, let $-\Delta_{\mathcal{C}}$ and  $-\Delta_{\mathbf E}$ be the Dirichlet Laplacian on $\mathcal{C}$ and $\mathbf E$, respectively. 

In this situation, the resonances of $P_\varepsilon$ are defined as the eigenvalues of the operator obtained by distorsion of  $P_\varepsilon$ in the complex plane with respect to the  coordinate $x$, for $x>L$ large enough.

We are interested in those resonances of $P_\varepsilon$ that are close to the eigenvalues of $-\Delta_{\mathcal C}$.

So, let $\lambda_0$ be an eigenvalue of $-\Delta_{\mathcal{C}}$, and let  $u_0$ be the corresponding (normalized) eigenfunction. For $k\geq 1$ integer, we set,
\be
\label{alfak}
\alpha_k:=k\pi/2;
\ee
the quantities $\alpha_k^2$ ($k\geq 1$) correspond to the thresholds of $-\Delta_{\mathbf E}$.

We assume,
\begin{itemize}
\item[] $\lambda_0$ is simple;
\item[\bf (H)\hskip 0.5cm] $\lambda_0>\alpha_1^2$ ; $\lambda_0\not= \alpha_k^2$ for all $k\geq 1$;
\item[] $u_0$ does not vanish on ${\mathcal C}$ near the point $(0,0)$.
\end{itemize}

Note that the first and  last property is automatically satisfied when $\lambda_0$ is the first eigenvalue of  $-\Delta_{\mathcal{C}}$. When $\lambda_0$ is a higher eigenvalues, it just means that $0$ is not on a nodal line of $u_0$.

Then, by the same arguments as in  \cite{HM}, we know that there is a resonance $\rho(\varepsilon)\in\mathbb{C}$ of  $P_\varepsilon$ such that $\rho(\varepsilon)\to \lambda_0$ as $\varepsilon\to 0$, and there is an eigenvalue $\lambda(\varepsilon)$ of  $P_{\mathcal{C}(\varepsilon)}$ such that, for all $\delta>0$, there is $C_\delta>0$, with, 
\be
|\rho(\varepsilon)-\lambda(\varepsilon)|\leq C_\delta e^{-\pi(1-\delta)L/\varepsilon},
\ee
for all $\varepsilon>0$ small enough.

 In particular, since $\lambda (\varepsilon)$ is real, an immediate consequence is,
\be
\label{upperbound}
|\im \rho(\varepsilon)|\leq C_\delta e^{-\pi(1-\delta)L/\varepsilon}.
\ee
 Here, we prove,
 \begin{theorem}\sl 
 \label{mainth}
 Under Assumption {\bf(H)}, there exists $N_0 >0$ such that, for all $\varepsilon >0$ small enough, one has,
 $$
 |\im \rho(\varepsilon)|\geq  \varepsilon^{N_0}e^{-\pi L/\varepsilon}.
 $$
 \end{theorem}
 \begin{remark}\sl
Following carefully the proof, one can see that  it is sufficient to take any  $N_0 >10.$
 \end{remark}
 The strategy of the proof  follows:
 \begin{itemize}
 \item By Green's formula, we  reduce to finding a lower-bound estimate on the resonant state $u_\varepsilon$ in the exterior domain $\mathbf E$;
 \item We find a representation of $u_\varepsilon$ by means of series on both sides of the aperture $\{L\}\times[-\varepsilon, \varepsilon]$;
 \item By matching the two representations  at the aperture, we  reduce to finding a lower-bound estimate on  $u_\varepsilon$ inside the neck ${\mathcal T}(\varepsilon)$;
 \item Then, using an argument from \cite{BHM}, the required estimate is deduced from an estimate on $u_0$ near $(0,0)$.
 \end{itemize}

\section{Properties of the resonant state}

By definition, the resonance $\rho (\varepsilon)$ is an eigenvalue of the complex distorted operator,
$$
P_\varepsilon (\mu):= U_\mu P_\varepsilon U_\mu^{-1},
$$
where $\mu>0$ is a small enough parameter, and $U_\mu$ is a complex distortion of the form,
$$
U_\mu \varphi (x,y):= \varphi (x+i\mu f(x), y),
$$
with $f\in C^\infty (\R)$, $f(x)= 0$ if $x\leq L+1$, $f(x) =x$ for $x$ large enough. (Observe that by the Weyl perturbation theorem, the essential spectrum of $P_\varepsilon (\mu)$ consists in the union of the half-lines $e^{-2i\theta}[\alpha_k^2, +\infty)$ ($k\geq 1$), with $\theta =\arctan\mu$.)

It is well known that such eigenvalues do not depend on $\mu$ (see, e.g., \cite{SZ, HeM}), and that the corresponding eigenfunctions are of the form $U_\mu u_\varepsilon$ with $u_\varepsilon$ independent of $\mu$, smooth on $\R^2$ and analytic with respect to $x$ in a complex sector around $\R$. Moreover, $u_\varepsilon$ can be normalized by setting, for some fixed $\mu>0$,
$$
\Vert U_\mu u_\varepsilon\Vert_{L^2(\Omega(\varepsilon )} =1.
$$

In that case, we learn from \cite{HM} (in particular Proposition 3.1 and formula (5.13)),
\be\label{ugrand}
\|u_\varepsilon\|_{L^2(\mathcal{C}\cup \mathcal{T}(\varepsilon)\cup (L,L+1)\times (-1,1))}\geq 1- \O(e^{(\delta-\pi/2  L)/\varepsilon}),
\ee and 
\be
\label{upetit}
\|u_\varepsilon\|_{H^1( (L,L+1)\times (-1,1))}=  \O(e^{(\delta-\pi/2  L)/\varepsilon}).
\ee
%Moreover, combining the results of \cite{BHM} and \cite{HM}, we see:

%For all $x\in(0,L)$ and  for all $\delta>0$, there exists $C>0$ such that, for all $\varepsilon >0$ small enough, one has,
%\be\label{utube}
%\|u_\varepsilon\|_{L^2([x,L]\times [-\varepsilon,\varepsilon] )}\geq C \varepsilon^{4.5+\delta} e^{-\pi x/2\varepsilon},
%\ee

Now, using the equation $-\Delta u_\varepsilon = \rho u_\varepsilon$ and the Green formula on the domain $\mathcal{C} \cup \left(  [L,L+1)\times (-1,1)\right)$, we obtain,

$$
\im \int_{-1}^1 u(L+1,y)\frac {\partial\overline u}{\partial x}(L+1,y)dy =\im\rho \int_{\mathcal{C}\cup \left ( [L,L+1)\times (-1,1)\right)  } |u|^2dxdy,
$$
and thus, by (\ref{ugrand}),
\be
 \label{green}
\im\rho =(1+\O(e^{(\delta-\pi/  L)/\varepsilon}))\,\,\im \int_{-1}^1 u(L+1,y)\frac {\partial\overline u}{\partial x}(L+1,y)dy .
\ee
 Therefore, in order to prove our result, it will be sufficient for us to obtain a lower bound on $\im \int_{-1}^1 u(L+1,y)\frac {\partial\overline u}{\partial x}(L+1,y)dy$. Note that, by using (\ref{upetit}), we immediately obtain (\ref{upperbound}).
 
\section{Representation in the thin tube}

Let $(\psi_k)_{k\geq 1}$ be an orthonormal basis of  eigenvectors of the Dirichlet realization of $-d^2/dy^2$ on $L^2(-\varepsilon, \varepsilon)$, and let $\alpha_k^2/{\varepsilon^2}$ be the corresponding eigenvalues. More precisely, for $k\geq 1$, we set,

\begin{eqnarray}
\label{psik}
&& \psi_{2k-1}(y):=\frac1{\sqrt \varepsilon}\cos(\alpha_{2k-1}\frac y \varepsilon);\\
&& \psi_{2k}(y):= \frac 1 {\sqrt \varepsilon} \sin( \alpha_{2k} \frac y \varepsilon).\nonumber
\end{eqnarray}

We also set,
$$
\theta_k:=\sqrt{\alpha_k^2-\varepsilon^2\rho(\varepsilon)},
$$
where $\sqrt{\cdot}$ stands for the principal square root, and we denote by $u_\varepsilon =u_\varepsilon (x,y)$ the resonant state of $P_\varepsilon$ corresponding to the resonance $\rho (\varepsilon)$, that is, the outgoing solution of the Dirichlet problem,
$$\left\{\begin{array}{c}
-\Delta u_\varepsilon = \rho (\varepsilon ) u_\varepsilon \mbox{ in } \Omega(\varepsilon);\\
u\left\vert_{\partial\Omega(\varepsilon)} \right. =0.
\end{array}\right.
$$

Then, for any $x\in (0, L)$, and for $\varepsilon$ small enough, we can expand  $u_\varepsilon (x,y)$ on the basis $(\psi_k(y))_{k\geq 1}$, 
$$
u_\varepsilon (x,y)=\sum_{k\geq 1} u_k(x) \psi_k(y),
$$
with,
$$
u_k(x):= \int_{-\varepsilon}^\varepsilon u_\varepsilon (x,y)\psi_k(y)dy.
$$
In particular, $u_k$ satisfies,
$$
\varepsilon^2 u''_k(x)=\theta_k^2u_k,
$$ 
and thus,  there exist $a_{k,+}, a_{k,-}\in\C$, such that, 
$$
 u_k(x)=a_{k,+}e^{\theta_k x/\varepsilon}+a_{k,-} e^{-\theta_k x/\varepsilon}.
$$
Therefore, we have proved that, for $x\in (0, L)$ and $\varepsilon$ small enough,
\begin{equation}\label{letube} 
u_\varepsilon(x,y)= \sum_{k=1}^{\infty} (a_{k,+} e^{\theta_k x/\varepsilon}+a_{k,-} e^{-\theta_k x/\varepsilon}) \psi_k(y),
\end{equation}
where the sum converges in $H^2((\delta, L-\delta)\times (-\varepsilon, \varepsilon))$ for any $\delta >0$.
Differentiating this identity with respect to $x$,  we also obtain, 
\begin{eqnarray}\label{letubederive} 
\frac {\partial u_\varepsilon}{\partial x} (x,y)=\frac 1 {\varepsilon}  \sum_{k=1}^{\infty} \theta_k( a_{k,+} e^{\theta_k x/\varepsilon}-a_{k,-} e^{-\theta_k x/\varepsilon}) \psi_k(y).
\end{eqnarray}

\section{Representation in the external tube}

In the same way as in the previous section, setting,

\begin{eqnarray}
\label{phik}
&& \varphi_{2j-1}(y):=\cos(\alpha_{2j-1}y);\\
&& \varphi_{2j}(y):=  \sin( \alpha_{2j} y),\nonumber
\end{eqnarray}
we can represent $u_\varepsilon$ in $\mathbf E$ as,
$$
u_\varepsilon (x,y) = \sum_{j\geq 1}v_j (x)\varphi_j (y),
$$
where the $v_j$'s satisfy
$$
v_j'' = (\alpha_j^2 - \rho (\varepsilon))v_j,
$$
and are thus of the form,
\be
\label{vj}
v_j(x) = b_{j,+} e^{(x-L)\sqrt{\alpha_j^2 - \rho } } + b_{j,-}e^{- (x-L)\sqrt{\alpha_j^2 - \rho } },
\ee
with $b_{j,\pm}\in \C$.
Note that, for $\im \rho <0$, our choice of the square root imposes that $\im \sqrt{\alpha_j^2 - \rho } \, >0$. 

By Assumption{\bf (H)},  there exists $j_0\geq 1$ such that,
 \be
 \label{hyplambda}
\alpha_{j_0}^2 < \lambda_0 <\alpha_{j_0+1}^2.
 \ee
In particular, for $\varepsilon$ small enough, $\alpha_{j_0}^2 <\re\rho(\varepsilon) <\alpha_{j_0+1}^2$. Moreover, by definition $u_\varepsilon$ must be out-going, and in our case this means that, for $\mu >0$ sufficiently small (but independent of $\varepsilon$), the distorted function $u_\varepsilon ((1+i\mu)x, y)$ is in $L^2([2L,+\infty)\times (-1,1))$. Therefore, in view of (\ref{vj}) and (\ref{hyplambda}), we necessarily have,
$$
b_{j,-} =0\;\mbox{for}\;(j\leq j_0)\quad b_{j,+}=0\; \mbox{for}\; j\geq j_0+1,
$$
that is, for $x>L$, $u_\varepsilon (x,y)$ can be re-written as,
\be
\label{uext}
u_\varepsilon (x,y) =\sum_{j\leq j_0} b_j e^{i(x-L)\sqrt{\rho -\alpha_1^2 } }\varphi_1(y) +\sum_{j\geq j_0+1}b_je^{- (x-L)\sqrt{\alpha_j^2 - \rho } }\varphi_j (y),
\ee
where, for any $L'>L$,  the series is absolutely convergent in $H^2((L,L')\times (-1,1))$.
 
 \section{Representation at the aperture}

Now, let us consider the trace $u_\varepsilon (L,y)$ of $u_\varepsilon$ on $\{ x=L\}$ (note that $u_\varepsilon$ is continuous on $\Omega(\varepsilon)$, and thus its trace is a well defined continuous function on $(-\varepsilon, \varepsilon)$). 

Since $u_\varepsilon\in H^2([L, L']\times[-1,1])$ ($L'>L$ arbitrary), and the part $\{(L,y)\,;\, |y|<1\}$ of the boundary of $[L, L']\times[-1,1]$ is smooth, by standard results (see, e.g., \cite{A,CP}), we know that $u_\varepsilon (L,y)$ is in $H^{3/2}_{\rm loc} (-1,1)$. But since it vanishes identically on $\{ \varepsilon <|y|<1\}$, we conclude,
\be
\label{H3demi}
u_\varepsilon (L,y)\in H^{3/2}([-1,1]).
\ee
On the other hand, on $\{ |y|<\varepsilon\}$, $u_\varepsilon (L,y)$ can be decomposed on the basis $(\psi_k)_{k\geq 1}$ as,
$$
u_\varepsilon (L,y)=\sum_{k\geq 1}C_k\psi_k(y),
$$
where the  $\{ C_k\}\in \ell^2(\N)$. Moreover, since $u_\varepsilon (L,\pm\varepsilon)=0$, if we denote by $\mathcal L(D_y)$ the Dirichlet realization of $-d^2/dy^2$ on $(-\varepsilon, \varepsilon)$, (\ref{H3demi}) implies,
$$
(\mathcal{L}(D_y)+1)^{3/4}u_\varepsilon (L,y) \in L^2(-\varepsilon, \varepsilon),
$$
and thus, using the fact that $\mathcal{L}(D_y)\psi_k = \alpha_k^2\psi_k$, and $\alpha_k\sim k$ as $k\to \infty$, we easily conclude,
\be
\label{Ckl3demi}
\sum_{k=1}^\infty k^3|C_k|^2 <\infty.
\ee
Now, with the notations of (\ref{letube}), we prove 
\begin{lemma}\sl
For all $k\geq 1$, one has,
$$
C_k = a_{k,+} e^{\theta_k L/\varepsilon}+a_{k,-} e^{-\theta_k L/\varepsilon}.
$$
\end{lemma}
\begin{proof}
By (\ref{letube}), it is enough to prove that the quantity,
$$
\Vert u_\varepsilon (L,\cdot) -u_\varepsilon (x,\cdot)\Vert_{L^2(-\varepsilon, \varepsilon)}
$$
tends to 0 as $x\to L_-$. This is probably a well known fact, but let us recall the proof. For $x<L$ and $|y|<\varepsilon$, we write,
$$
u_\varepsilon (L,y) -u_\varepsilon (x,y)=\int_x^L\partial_xu_\varepsilon (t,y)dt,
$$
and thus, by Cauchy-Schwarz inequality,
$$
\Vert u_\varepsilon (L,\cdot) -u_\varepsilon (x,\cdot)\Vert_{L^2(-\varepsilon, \varepsilon)}\leq \sqrt{(L-x)}\Vert \partial_xu_\varepsilon\Vert_{L^2((0,L)\times (\varepsilon, \varepsilon))}.
$$
Denoting by $u_\varepsilon^\mu (x,y):=u_\varepsilon (x+i\mu f(x),y)$ the function obtained by distorting $u_\varepsilon$ in the complex in $x$ , we also have,
$$
\Vert \partial_xu_\varepsilon\Vert_{L^2((0,L)\times (\varepsilon, \varepsilon))}=\Vert \partial_xu^\mu_\varepsilon\Vert_{L^2((0,L)\times (\varepsilon, \varepsilon))}\leq \Vert \nabla u^\mu_\varepsilon\Vert_{L^2(\Omega(\varepsilon))}={\mathcal O}(1),
$$
and the result follows.
\end{proof}

Therefore, for $|y|<\varepsilon$, we have proved,
\be
\label{match1}
u_\varepsilon (L,y)=\sum_{k\geq 1}\left( a_{k,+} e^{\theta_k L/\varepsilon}+a_{k,-} e^{-\theta_k L/\varepsilon}\right)\psi_k(y),
\ee
and
\be
\sum_{k\geq 1}k^3\left| a_{k,+} e^{\theta_k L/\varepsilon}+a_{k,-} e^{-\theta_k L/\varepsilon}\right|^2 <\infty.
\ee

In the same way, by taking the limit $x\to L_+$, for $|y|<1$ we also obtain,
\be
u_\varepsilon (L,y)=\sum_{j\geq 1}b_j\varphi_j (y),
\ee
and,
\be
\sum_{j\geq 1}j^3|b_j|^2<\infty.
\ee

Similar arguments can be performed for the derivative $\partial_x u_\varepsilon \in H^1(\Omega(\varepsilon)\cap \{|y|<L'\}) $, and they lead to,
\be
\partial_x u_\varepsilon (L,y) = \frac1{\varepsilon}\sum_{k\geq 1}\theta_k\left( a_{k,+} e^{\theta_k L/\varepsilon}-a_{k,-} e^{-\theta_k L/\varepsilon}\right)\psi_k(y)\, \mbox{ in } H^{1/2}(|y|\leq\varepsilon);
\ee
and
\begin{eqnarray}
 \partial_x u_\varepsilon (L,y) = i\sum_{j\leq j_0}( \sqrt{\rho -\alpha_j^2 })b_j\varphi_j(y) -\sum_{j>j_0}(\sqrt{\alpha_j^2 - \rho })b_j \varphi_j (y)\nonumber\\ 
 \label{match2}\, \mbox{ in } H^{1/2}(|y|\leq 1).
\end{eqnarray}

\section{Estimates on the coefficients }

In this section, taking advantage of the two previous representations of $u_\varepsilon$ at the aperture, we compute in two different ways the three following quantities:

$$
\la u_\varepsilon, \partial_x u_\varepsilon\ra_{\{L\}\times[-1,1]}\, ;\, \la u_\varepsilon,\varphi_1\ra_{\{L\}\times[-1,1]}\, ;\, \la\partial_x u_\varepsilon,\psi_1\ra_{\{L\}\times[-\varepsilon,\varepsilon]}.
$$

The resulting identities will permit us to give a lower bound on $\displaystyle\sum_{j\leq j_0}|b_j|^2$ in terms of $|a_{1,-}|$ and to conclude by using an argument from \cite{BHM}.

From now on, we set,
$$
A_{k,\pm}:= a_{k,\pm}e^{\pm \theta_k L/\varepsilon}.
$$

Since  $u_\varepsilon (L,y)$  vanishes identically on $\{ |y| >\varepsilon\}$, in view of (\ref{match1})-(\ref{match2}), the two computations of  $\la u_\varepsilon, \partial_x u_\varepsilon\ra_{\{L\}\times[-1,1]}$ give the identity,
\begin{eqnarray*}
&& \frac1{\varepsilon}\sum_{k\geq 1}\theta_k(|A_{k,+}|^2-|A_{k,-}|^2+2i\im (A_{k,+}\overline{A}_{k,-}))\\
&&\hskip 4cm = \sum_{j\leq j_0} i( \sqrt{\rho -\alpha_j^2 })|b_j|^2-\sum_{j\geq j_0+1}(\sqrt{\alpha_j^2 - \rho })|b_j|^2.
\end{eqnarray*}

Then, using the fact that $\re\theta_k \sim k$ as $k\to\infty$, while $|\im\theta_k|={\mathcal O}(k^{-1}e^{-\delta /\varepsilon})$ and $|\im \sqrt{\rho -\alpha_j^2}|={\mathcal O}(e^{-\delta /\varepsilon})$ for some constant $\delta >0$, and all $j\leq j_0$ by taking the real part we deduce,

\begin{eqnarray*}
&& \frac1{\varepsilon}\sum_{k\geq 1}(\re\theta_k)(|A_{k,+}|^2-|A_{k,-}|^2) +\frac1{\varepsilon}\sum_{k\geq 1}{\mathcal O}(k^{-1}e^{-\delta/\varepsilon}) |A_{k,+}A_{k,-}|)\\
&&\hskip 4cm = {\mathcal O}(e^{-\delta /\varepsilon})\sum_{j\leq j_0} |b_j|^2-\sum_{j\geq j_0+1}(\re\sqrt{\alpha_j^2 - \rho })|b_{j}|^2.
\end{eqnarray*}

In particular, since $\re\sqrt{\alpha_j^2 - \rho }= \frac{\pi  j}2(1+{\mathcal O}(\varepsilon^2 j^{-2}))$, we see that there exists a constant $C>0$ such that,

\begin{eqnarray}\displaystyle
\label{est1}
 && \sum_{k\geq 1}\re\theta_k(|A_{k,+}|^2-|A_{k,-}|^2) 
\leq Ce^{-\delta /\varepsilon}\sum_{j\leq j_0} |b_j|^2\\
 && \nonumber\hskip 1cm +C\sum_{k\geq 1}k^{-1}e^{-\delta/\varepsilon} |A_{k,+}A_{k,-}| 
-\frac{\pi}2\varepsilon \sum_{j\geq j_0+1} j(1-C\varepsilon^2j^{-2})|b_{j}|^2.
\end{eqnarray}

Moreover, by  (\ref{a-}), we see that,
\be
\label{Ak2}
\sum_{k\geq 2} k| A_{k,-}|^2={\mathcal O}(\varepsilon^{-1/2}e^{-2\pi L/\varepsilon}),
\ee
and we also know from (\ref{estb})
$\sum_{j\leq j_0}|b_{j}|^2 ={\mathcal O}(e^{(\delta' - \pi L)/\varepsilon})$ for any $\delta'>0$.
Therefore, we deduce from (\ref{est1})(with other positive constants $C,\delta$), 
\begin{eqnarray}
\label{est1'}
 && \sum_{k\geq 1}(k-Ck^{-1}e^{-\delta/\varepsilon})|A_{k,+}|^2  \\
 \nonumber
 && \leq  (1+Ce^{-\delta /\varepsilon}) |A_{1,-}|^2+ Ce^{-(\pi L +\delta) /\varepsilon}-\varepsilon\sum_{j\geq j_0+1} j(1-C\varepsilon^2j^{-2})|b_{j}|^2;
\end{eqnarray}

Now,  computing the scalar products   $\la u_\varepsilon (L, \cdot),\varphi_1\ra$  and $\la \partial_xu_\varepsilon (L, \cdot), \psi_1\ra_{L^2(|y|<\varepsilon)}$ in two different ways (by using (\ref{match1})-(\ref{match2}) and the fact that $u_\varepsilon (L,y)=0$ on $\{\varepsilon < |y| <1\}$), we find,
\begin{eqnarray*}
&&\sum_{k\geq 1} \mu_k(A_{k,+}+A_{k,-}) =b_{1};\\
&& \frac1{\varepsilon}\theta_1(A_{1,+}-A_{1,-}) =\sum_{j\leq j_0} i\nu_j( \sqrt{\rho -\alpha_j^2 })b_{j} -\sum_{j\geq j_0+1}\nu_j(\sqrt{\alpha_j^2 - \rho })b_{j},
\end{eqnarray*}
with,
$$
\mu_k:=\int_{-\varepsilon}^\varepsilon \psi_k(y)\varphi_1(y)dy =\left\{\begin{array}{l}
0 \mbox{ if $k$ is even};\\
(-1)^{\frac{k-1}2}\frac{4k\sqrt\varepsilon}{\pi (k^2-\varepsilon^2)} \cos\frac{\varepsilon\pi}2 \, \mbox{ if $k$ is odd},
\end{array}\right.
$$
and,
$$
\nu_j:= \int_{-\varepsilon}^\varepsilon \varphi_j(y)\psi_1(y)dy =\left\{\begin{array}{l}
0 \mbox{ if $j$ is even};\\
\frac{4\sqrt{\varepsilon}\sin((\varepsilon j -1)\pi/2)}{\pi (\varepsilon^2j^2-1)} \mbox{ if $j\not= \frac1{\varepsilon}$ is odd};\\
\sqrt{\varepsilon} \mbox{ if $j=\frac1{\varepsilon}$ is odd}.
\end{array}\right.
$$
In particular, using the fact that  $|\sin t|\leq \min (|t|, 1)$,
$$
\nu_j\leq \frac{4 \sqrt{\varepsilon}}{\pi(\varepsilon j +1)}\min(\frac{\pi}2, \frac1{|\varepsilon j -1|}),
$$
 and thus, using again (\ref{Ak2}), we obtain,
\begin{eqnarray}
\label{A+B}
&& |A_{1,+} + A_{1,-}| \leq \frac{C_0}{\sqrt{\varepsilon}}|b_{1}| + \sum_{k\geq 2}|\frac{\mu_k }{\mu_1}A_{k,+}| +\frac{C_0}{\sqrt{\varepsilon}}e^{-\pi L/\varepsilon};\\
\label{A-B}
&& | A_{1,+} - A_{1,-}| \leq  C_0\varepsilon^{\frac32}\sum_{j\leq j_0} |b_{j}|  +\frac{C_0}{\sqrt{\varepsilon}}e^{-\pi L/\varepsilon}\\
 &&\hskip 2cm \nonumber+\frac{4\varepsilon^{\frac32}}{\pi |\theta_1|} \sum_{j\geq j_0+1}\frac{|\alpha_j|}{\varepsilon j+1}\min(\frac{\pi}2,\frac1{|\varepsilon j-1|}) |b_{j}|,
\end{eqnarray}
with some constant $C_0>0$.

Then, we observe that $|\mu_k/\mu_1|\leq (k-\varepsilon^2)^{-1}$, and thus, by (\ref{est1'}),
\begin{eqnarray}\nonumber
\sum_{k\geq 2}|\frac{\mu_k }{\mu_1}A_{k,+}|\leq \left(\sum_{k\geq 2}\frac1{(k-\varepsilon^2)^2}\right)^{\frac12}\left(\sum_{k\geq 2}|A_{k,+}|^2\right)^{\frac12}\\
\label{estsom1}
\leq \tau_1  \left(|A_{1,-}|^2-\varepsilon\sum_{j\geq j_0+1}j|b_{j}|^2\right)^{\frac12} + Ce^{-(\pi L +\delta) /2\varepsilon},
\end{eqnarray}
where $\tau_1$ can be taken arbitrarily close to $\frac12(\sum_{k\geq 2}k^{-2})^{\frac12}\leq \frac{\sqrt{3}}4<\frac12$. Inserting (\ref{estsom1}) into (\ref{A+B}), we obtain,
\begin{eqnarray}
\label{A+B1}
&&|A_{1,+} + A_{1,-}| \leq \tau_1   \left(|A_{1,-}|^2-\varepsilon\sum_{j\geq j_0+1}j|b_{j}|^2\right)^{\frac12}\\ &&\hskip 1cm \nonumber
+ \frac{C_0}{\sqrt{\varepsilon}}|b_{1}|  +2Ce^{-(\pi L +\delta) /2\varepsilon}.
\end{eqnarray}

On the other hand, setting $\gamma_0:=1+\frac2{\pi}\sim 1.637$, and using the fact that
 $|\alpha_j|\leq j\pi/2$ while $|\theta_1| =\frac{\pi}2 (1+{\mathcal O}(\varepsilon^2)$, we have,
\begin{eqnarray*}
 \frac{4\varepsilon^{\frac32}}{\pi |\theta_1|} \sum_{j\geq j_0+1}\frac{|\alpha_j|}{\varepsilon j+1}\min(\frac{\pi}2,\frac1{|\varepsilon j-1|}) |b_{j}|\leq 
 (1+C\varepsilon^2) \sqrt{\varepsilon}\times\\
 (2\sum_{j_0+1\leq j\leq\gamma_0/\varepsilon}\frac{\varepsilon j}{\varepsilon j +1}|b_{j}|
 + \frac4{\pi}\sum_{j\geq\gamma_0/\varepsilon}\frac{\varepsilon j}{\varepsilon^2 j^2-1}|b_{j}|).
\end{eqnarray*}
Therefore, by the Cauchy-Schwarz inequality,
\begin{eqnarray}
\nonumber
 && \frac{4\varepsilon^{\frac32}}{\pi |\theta_1|} \sum_{j\geq j_0+1}\frac{|\alpha_j|}{\varepsilon j+1}\min(\frac{\pi}2,\frac1{|\varepsilon j-1|}) |b_{j}|\\
 \label{A-B2}
 && \hskip 3cm \leq (1+C\varepsilon^2)\sqrt{ 4\Gamma_1+\frac 6 {\pi^2} \Gamma_2}\left(\varepsilon\sum_{j\geq j_0+1}j|b_{j}|^2\right)^{\frac12},
 \end{eqnarray}
 with,
 $$
 \Gamma_1:=\sum_{j_0+1\leq j\leq\gamma_0/\varepsilon}\frac{\varepsilon^2 j}{(\varepsilon j +1)^2},
\mbox{ and }
 \Gamma_2:=\sum_{j\geq\gamma_0/\varepsilon}\frac{\varepsilon^2 j}{(\varepsilon^2 j^2-1)^2}.
 $$
As $\varepsilon \to 0$, we see that $\Gamma_1$ tends to 
$$
I_1:=\int_0^{\gamma_0}\frac{tdt}{(t+1)^2} = \ln (1+\gamma_0)-\frac{\gamma_0}{1+\gamma_0}\sim 0.97 -0.62=0.35,
$$
while $\Gamma_2$ tends to
$$
I_2:=\int_{\gamma_0}^\infty \frac{tdt}{(t^2-1)^2}=-\frac12 \left[\frac1{t^2-1}\right]_{\gamma_0}^\infty=\frac1{2(\gamma_0^2-1)}\sim 0.298.
$$
Therefore, we deduce from (\ref{A-B}) and (\ref{A-B2}),
\be
\label{A-B1}
 | A_{1,+} - A_{1,-}| \leq  C_0\varepsilon^{\frac32} \sum_{j\leq j_0} |b_{j}|  +\tau_2 \left(\varepsilon\sum_{j\geq j_0+1}j|b_{j,-}|^2\right)^{\frac12}+\frac{C_0}{\sqrt{\varepsilon}}e^{-\pi L/\varepsilon},
\ee
where $\tau_2$ can be taken arbitrarily close to 
$$\sqrt{ 4\Gamma_1+\frac 6 {\pi^2} \Gamma_2} \leq 1.6.$$
 Summing (\ref{A+B1}) with (\ref{A-B1}), and using the triangle inequality, we finally obtain,
\be
\label{estA1}
2|A_{1,-}|\leq \tau_1   \sqrt{|A_{1,-}|^2-X}+\tau_2\sqrt{X}+\sum_{j\leq j_0} \frac{2C}{\sqrt{\varepsilon}}|b_{j}|  +3Ce^{-(\pi L +\delta) /2\varepsilon},
\ee
where we have set,
$$
X:=\varepsilon\sum_{j\geq j_0+1}j|b_{j}|^2.
$$
Now, an elementary computation shows that the map,
$$
[0,|A_{1,-}|^2]\ni X\mapsto \tau_1   \sqrt{|A_{1,-}|^2-X}+\tau_2\sqrt{X}
$$
 reaches its maximum at $X=\frac{\tau_2^2}{\tau_1^2+\tau_2^2}|A_{1,-}|$, and the maximum value is
 $$
 (\sqrt{\tau_1^2+\tau_2^2})|A_{1,-}|.
 $$
  Therefore, we deduce from (\ref{estA1}),
  \be
  2|A_{1,-}|\leq  (\sqrt{\tau_1^2+\tau_2^2})|A_{1,-}|+\sum_{j\leq j_0} \frac{2C}{\sqrt{\varepsilon}}|b_{j}|   +3Ce^{-(\pi L +\delta) /2\varepsilon},
  \ee
Since $\tau_1^2+\tau_2^2\leq 3 <4$, we have proved,
\begin{proposition}\label{atob}\sl
There exist two constants $C, \delta >0$ such that, for any $\varepsilon >0$ small enough, one has,
$$
|A_{1,-}|\leq \frac{C}{\sqrt\varepsilon}\sum_{j\leq j_0} |b_{j}|+Ce^{-(\pi L +\delta) /\varepsilon}.
$$
\end{proposition}

\section{ Completion of the proof}

We first observe,
 \begin{proposition}\sl \label{imrho}There exists a constant $C_0>0$, such that,
$$|\im(\rho)|\geq \frac1{C_0}(\sum_{j=1}^{j_0} |b_j |)^2,$$ 
for all $\varepsilon>0$ small enough. 
\end{proposition}
\begin{proof}

Let us compute $\im\big( \int_{-1}^{1} u(L+1,y)\frac {\partial\overline u(L+1,y)}{\partial x} dy \big)$ with the help of the expression (\ref{uext}). We first obtain,

\begin{eqnarray*}
&&\frac {\partial \overline u_\varepsilon}{\partial x} (x,y) = \sum_{j\leq j_0}-i\overline{ \sqrt{\rho -\alpha_j^2 } }\;b_j e^{-i(x-L) \overline{ \sqrt{\rho -\alpha_j^2 }}}\; \overline{ \varphi_1(y)} \\
&&\hskip 4cm-\sum_{j\geq j_0+1}b_je^{- (x-L)\overline{\sqrt{\alpha_j^2 - \rho }}}\;\overline{\varphi_j (y)}  \quad (x > L),
\end{eqnarray*}
and thus,
\begin{eqnarray*}
  \int_{-1}^{1} u(L+1,y)\frac {\partial\overline u(L+1,y)}{\partial x} dy &=&
-\sum_{j\leq j_0}i |b_j |^2       \overline{ \sqrt{\rho -\alpha_j^2 } } e^{-2 \, \im{ \sqrt{\rho -\alpha_j^2 }}}\\
&& \hskip 0.5cm- \sum_{j\geq j_0+1}   |b_j|^2  \;\overline{ \sqrt{\alpha_j^2 -\rho} } \;e^{-2 \re{\sqrt{\alpha_j^2 - \rho }}}.
\end{eqnarray*}

Taking the imaginary part, we obtain,
\begin{eqnarray*}
|\im { \int_{-1}^{1} u(L+1,y)\frac {\partial\overline u(L+1,y)}{\partial x} dy} |      \geq \sum_{j\leq j_0}  |b_j |^2     
  \re{ \sqrt{\rho -\alpha_j^2 } }     e^{-2  \im{ \sqrt{\rho -\alpha_j^2 } } }\\
 -  \sum_{j\geq j_0+1}   |b_j|^2   \im { \sqrt{\alpha_j^2 -\rho} }  e^{-2 \re{\sqrt{\alpha_j^2 - \rho }}}.
 \end{eqnarray*}
Now, for $j\geq j_0+1$, we have 
$| \im { \sqrt{\alpha_j^2 -\rho} } |\leq C  |\im{\rho}|$,  while, for $j\leq j_0$, there exist $c_0, C_0 >0$, such that,
$ 2c_0\leq  \re{ \sqrt{\rho -\alpha_j^2 } }   \leq C_0$, and $| \im { \sqrt{\rho -\alpha_j^2 } } |\leq C  |\im{\rho}|$.

Then, from (\ref{estb}), we obtain 
$$|\im {  \int_{-1}^{1} u(L+1,y)\frac {\partial\overline u(L+1,y)}{\partial x} dy  }|  \geq   c_0 \sum_{j\leq j_0}  |b_j |^2 
 -  e^{(\delta'-\pi L)/\varepsilon} | \im{\rho}|$$

The equation (\ref{green})  combined with the  previous estimate gives,
 $$|\im(\rho)| (1+\O(e^{(\delta'-\pi  L)/\varepsilon}))\geq c_0 \sum_{j\leq j_0} |b_j |^2 ,
 $$
and since $\sum_{j\leq j_0} |b_j |^2\geq j_0^{-2}(\sum_{j\leq j_0} |b_j |)^2$, the result follows.
\end{proof}
 
In view of Propositions \ref{atob} and \ref{imrho}, we see that it only remains to find an appropriate lower bound on $|A_{1,-}|$. This will be achieved by using an argument from \cite{BHM}.

 Indeed, by Assumption {\bf (H)}, we see that the Dirichlet eigenfunction $u_0$ satisfies the hypothesis of \cite{BHM} Lemma 3.1. Then, following the arguments of \cite{BHM} leading to (13) in that paper, and using again \cite{HM}, Proposition 3.1 and Formula (5.13), we conclude that for any $\delta >0$ and any $x\in (0,L)$, there exists $C_1$ such that the resonant state $u_\varepsilon$ verifies (see \cite{BHM}, Formula(13)),
\be\label{utube}
\|u_\varepsilon\|_{L^2([x,L]\times [-\varepsilon,\varepsilon] )}\geq C_0 \varepsilon^{4.5+\delta} e^{-\pi x/2\varepsilon}.
\ee
Thanks to this estimate, we can prove,

\begin{proposition}\sl\label{step4} For any $\delta >0$, there exists $C>0$, such that,
\begin{eqnarray}  
   |A_{1,-}| \geq   C  \varepsilon^{4.5+\delta}   e^{- \pi L/2\varepsilon},\end{eqnarray}
   for $\varepsilon >0$ small enough.
   \end{proposition}
   \begin{proof}
The estimation (\ref{est1'}) also gives,

\begin{eqnarray}
\label{estA+}
 && \sum_{k\geq 1} |A_{k,+}|^2   \leq  (1+Ce^{-\delta /\varepsilon}) |A_{1,-}|^2+ Ce^{-(\pi L +\delta) /\varepsilon} 
\end{eqnarray}
Let compute the quantity $\|u_\varepsilon\|_{L^2([x,L]\times [-\varepsilon,\varepsilon] )}$ 
using the expression (\ref{letube}), for any $x$ fixed,

\begin{eqnarray}
\nonumber \|u_\varepsilon\|^2_{L^2([x,L]\times [-\varepsilon,\varepsilon] )}
&=&\sum_{k\geq 1} |a_{k,+}|^2 \frac {\varepsilon}{2\re{\theta_k}} 
 \big( e^{L\;2\re{\theta_k}/\varepsilon }-e^{x \;2\re{\theta_k}/\varepsilon}\big)\\
 \nonumber
 &+&\sum_{k\geq 1} 2 \re{\Big( \frac {\varepsilon a_{k,+}\overline a_{k,-} }{2i\im{\theta_k}} 
\big( e^{iL\;2\im{\theta_k}/\varepsilon}-e^{ix \;2\im{\theta_k} /\varepsilon} \big)\Big)}\\
\nonumber &+&\sum_{k\geq 1} |a_{k,-}|^2 \frac {\varepsilon}{2\re{ \theta_k}}
\big( e^{-x 2\re{\theta_k}/\varepsilon}-e^{-L 2\re{\theta_k}/\varepsilon}  \big).
\end{eqnarray}

So this give the inequality 
\begin{eqnarray}
&&\nonumber \|u_\varepsilon\|^2_{L^2([x,L]\times [-\varepsilon,\varepsilon] )}\leq 
2\sum_{k\geq 1} |A_{k,+}|^2 
+\sum_{k\geq 1} 2 |A_{k,+}|  |a_{k,-}|  e^{-L  \re{\theta_k} /\varepsilon }   \\
\nonumber 
&&+\sum_{k> 1} |a_{k,-}|^2  \varepsilon
 e^{-x \;2\re{\theta_k}/\varepsilon} 
 +|a_{1,-}|^2  \frac {\varepsilon}{2\re{ \theta_1}}
 e^{-x  \;2\re{\theta_1}/\varepsilon} .\end{eqnarray}

By the Cauchy-Schwarz inequality, 

\begin{eqnarray}
&&\nonumber \|u_\varepsilon\|^2_{L^2([x,L]\times [-\varepsilon,\varepsilon] )}\leq 
4\sum_{k\geq 1} |A_{k,+}|^2  \\
\nonumber &&+4\sum_{k> 1} |a_{k,-}|^2 
 e^{-x \;2\re{\theta_k}/\varepsilon } 
 +|a_{1,-}|^2  \varepsilon
 e^{-x\;2\re{\theta_1}/\varepsilon }. \end{eqnarray}

Using  (\ref{estA+})  and (\ref{a-}) we obtain 
\begin{eqnarray} \label{step1}
&& \|u_\varepsilon\|^2_{L^2([x,L]\times [-\varepsilon,\varepsilon] )}\leq C |a_{1,-}|^2  
 e^{-x \;2\re{\theta_1}/  \varepsilon} \\
 &&\nonumber \hskip 3cm+C \varepsilon    e^{-(\pi L +\delta) /\varepsilon} +
 C \varepsilon ^{-C} e^{-x \;2\re{\theta_1}/\varepsilon} e^{-x2C_0/\varepsilon}. \end{eqnarray}

Now using (\ref{utube})   we get 

\begin{eqnarray} \label{step3} 
 \varepsilon^{c} (1- \varepsilon ^{-c}  e^{-x2C_0/\varepsilon} -C  \varepsilon ^{-c}    e^{-(\pi (L-x) +\delta) } )  \leq  |a_{1,-}|^2  
\end{eqnarray}
with $c:=9+2\delta$.
Thus, for $\varepsilon$ small enough, we obtain
 \begin{eqnarray}  
C  \varepsilon^{c}   \leq  |a_{1,-}|^2  .
\end{eqnarray}
\end{proof}

Combining the results of Propositions \ref{atob}, \ref{imrho} and \ref{step4}, our main result Theorem \ref{mainth} follows.

\appendix

\section{}
\label{appA}

Using the equation
$P_\varepsilon u_{\varepsilon} =\rho(\varepsilon) u_\varepsilon,$
 we obtain,
$$
\|u_\varepsilon \|_{L^2({\mathcal {C}\cup \mathcal {T}(\varepsilon)})} + \|\Delta u_\varepsilon \|_{L^2({\mathcal {C}\cup \mathcal {T}(\varepsilon) })}=\O(1),
$$
 uniformly in $\varepsilon$.

Since $u_{\varepsilon}\in H^2(\mathcal {C}\cup \mathcal {T}(\varepsilon))$ the trace theorem  applies at  $x=c\varepsilon$ with $c>0$ sufficiently large, and a   scaling proves that,

$$\|u_\varepsilon(c\varepsilon,y)\|_{H^{1/2}(-\varepsilon,\varepsilon)}= \O(\varepsilon^{-1/2})$$
$$\|\frac {\partial u_\varepsilon}{\partial x} (c\varepsilon,y)\|_{L^2(-\varepsilon,\varepsilon)}= \O( 1).$$

In particular,
$$
\Vert(\mathcal{L}(D_y)+1)^{1/4}u_\varepsilon (c\varepsilon,y) \Vert_{L^2}=\O(\varepsilon^{-1/2}).
$$
and, using the same argument as for (\ref{Ckl3demi}), we
 easily conclude, 
\be \nonumber  \sum_{k\geq 1}  k | a_{k,+} e^{c\theta_k} + a_{k,-}e^{-c\theta_k}|^2 =\O( \varepsilon ^{-1/2} );\ee
\be \nonumber   \sum_{k\geq 1} k^2  |a_{k,+}e^{c\theta_k} - a_{k,-}  e^{-c\theta_k}  |^2 =\O( 1).\ee

We deduce,
\be \label{a+}\sum_{k\geq 1} k |a_{k,+} e^{c\theta_k} |^2 =\O( \varepsilon ^{-1/2} ), \ee
and 
\be\label{a-}\sum_{k\geq 1}  k  |a_{k,-} e^{-c\theta_k} |^2 =\O( \varepsilon ^{-1/2} ).\ee

\section{}
\label{appB}

This appendix is devoted to the proof of the estimate,

\begin{eqnarray} \label{estb}\sum_{j\leq j_0} |b_j|^2+ \sum_{j\geq j_0+1} |b_j|^2 e^{ -\re(\sqrt{\alpha_j^2 - \rho }) }=\O(e^{(\delta-\pi  L)/\varepsilon}).\end{eqnarray}

Using the expression (\ref{uext}), we compute, 
\begin{eqnarray*}
\|u_\epsilon\|^2_{L^2((L,L+1)\times (-1,1))}
&=& \sum_{j\leq j_0} |b_j|^2 \int_{(L,L+1)} e^{-2(x-L)\im(\sqrt{\rho -\alpha_j^2 }) } dx \\
&&\hskip .3cm+\sum_{j\geq j_0+1}|b_j|^2 \int_{(L,L+1)} e^{-2 (x-L)\re(\sqrt{\alpha_j^2 - \rho }) }dx,
\end{eqnarray*}
and thus,
\begin{eqnarray*}\|u_\epsilon\|^2_{L^2((L,L+1)\times (-1,1))}&\geq &\frac1{C}\sum_{j\leq j_0} |b_j|^2+\frac1{C}\sum_{j\geq j_0+1} |b_j|^2 e^{ -\re(\sqrt{\alpha_j^2 - \rho }) },
\end{eqnarray*}
for some positive constant $C$.
With the inequality (\ref{upetit}), this gives (\ref{estb}).

\end{document}